\documentclass[a4paper]{article}

 \usepackage{amsthm}
\usepackage{amsfonts}
\usepackage{bbm}

\usepackage{subfig}
\usepackage{flushend}
\usepackage{indentfirst}
\usepackage{graphics}
\usepackage{amsmath}
\usepackage{graphicx}
\usepackage{epstopdf}
\usepackage{float}

\usepackage[font=footnotesize,labelfont=bf]{caption}
\usepackage[labelsep=period]{caption}

\graphicspath{{./figure/}}

\newtheorem{thm}{Theorem}

\begin{document}

\title{ {Modelling with given reliability and accuracy in the space ${{L}_{p}}(T)$ of stochastic processes from $Sub_\varphi(\Omega)$ decomposable in series with independent elements}}

\date{}

\maketitle
Applied Statistics. Actuarial and Financial Mathematics.  No. 2, 13--23,  2012

\vspace{20pt}

\author{\textbf{Oleksandr Mokliachuk}$^{1,*}$\\\\
\footnotesize $^{1}${National Technical University of Ukraine ``Igor Sikorsky Kyiv Polytechnic Institute'', Kyiv, Ukraine}\\
\footnotesize $^{*}$Corresponding email: omoklyachuk@gmail.com}\\\\\\


\noindent \textbf{\large{Abstract}} \hspace{2pt}
Models that approximate stochastic processes from $Sub_\varphi(\Omega)$ with given reliability and accuracy in ${{L}_{p}}(T)$ for some given $\varphi(t)$ are considered.
We also study construction of models of processes which can be decomposed into series with approximate elements.
Karhunen-Lo{\`e}ve model is considered as an example of the application of the proposed construction.
\\

\noindent\textbf{{Keywords}} \hspace{2pt}
 Models of stochastic processes, Karhunen-Lo{\`e}ve model, sub-Gaussian processes.
Models of stochastic processes,
\\


\noindent\hrulefill

\section{{Introduction}}

Let $\left(\Omega,{\cal F}, P \right)$ be a standard probability space, let $L^{o}_2(\Omega)$  be the space of centered random variables with finite second moment,
$E\xi=0$, $E\xi^2<\infty$, and let $\{\Lambda, {\cal U}, \mu\}$ be a measurable space with a $\sigma$-finite measure $\mu$.
Let $L_p(\Lambda,\mu)$ be a Banach space of integrable to the power $p$  functions with the measure $\mu $.

Definition 1. \cite{Buldygin}
A random variable $\xi$ is called sub-Gaussian if there exist $a\geq 0$, such
that for all $\lambda\in R$ the following inequality holds true:
$$E\exp\{\lambda\xi\}\leq\exp\left\{\frac{a^2\lambda^2}{2}\right\}.$$
The characteristic of the random variable $\xi$, specified as
$$\tau(\xi)=\inf\left\{a\geq0: E \exp\{\lambda\xi\}\leq \exp\left\{\frac{a^2\lambda^2}{2}\right\},\lambda\in  R\right\}$$
will be called a sub-Gaussian standard of the random variable $\xi$.

Definition 2. \cite{Buldygin} A continuous even convex function $\varphi=\{\varphi(x),x\in R\}$ is called $N$--Orlicz function,
if it increases in the domain $x>0$, $\varphi(0)=0$, $\varphi(x)>0$ for $x\neq0$ and the following conditions hold true:
$$\lim_{x\to 0} \frac{\varphi(x)}{x}=0, \quad \text{and}\quad
\lim_{x\to \infty} \frac{\varphi(x)}{x}=\infty.$$

Definition 3. \cite{Buldygin} Let $\varphi=\{\varphi(x),x\in R\}$ be an $N$--Orlicz function. The function
$$\varphi^*(x)=\sup_{y\in R}(xy-\varphi(y))$$ is called the Young-Fenchel transform of the function $\varphi$.

Definition 4. \cite{Buldygin} Let $\varphi$ be an  Orlicz function such that
$$\lim_{x\to 0} \inf_x \frac{\varphi(x)}{x^2}=c>0.$$
A random variable $\xi$ belongs to the space $Sub_\varphi(\Omega)$, if $E\xi=0$, $E\exp\{\lambda\xi\}$ exists for all $\lambda\in R$,
and there exists a number $a>0$ such that for all $\lambda\in R$ the inequality $$E\exp\{\lambda\xi\}\leq \exp\{\varphi(\lambda a)\}$$ is true.

Definition 5. \cite{kozach} Let $X=\{X(t),t\in [0,T]\}$ be a stochastic process from the space $Sub_\varphi(\Omega)$.
A stochastic process $X_N=\{X_N(t),t\in [0,T]\}$ from the space $Sub_\varphi(\Omega)$ we will call a model that
approximates a stochastic process $X$  with given reliability $1-\alpha$ and accuracy $\delta$ in the space $L_p(0,T)$, if
$$P\left\{\int_0^T( X(t)-X_N(t))^pdt)^{1/p}>\delta\right\}\leq\alpha.$$

The following theorem is proved in \cite{koz-kam}.

\begin{thm}\label{thm1}
 \cite{koz-kam} Let $\{\mathbb T,\Lambda, M\}$ be a measurable space, let $X=\{X(t),t\in\mathbb T\} \subset S u b_\varphi (\Omega)$, and let $\tau_\varphi(t)=\tau_\varphi(X(t))$.
Let the integral
$$\int_{\mathbb T} (\tau_\varphi (t))^p d\mu (t)=c<\infty$$   exist.
Then the integral $$\int_{\mathbb T} |X(t)|^p d\mu(t)<\infty$$ exists with probability 1 and
$$P\left\{\int_{\mathbb T} |X(t)|^p d\mu(t) >\delta\right\} \leq 2\exp\left\{-\varphi^* \left(\left(\frac{\delta}{c}\right)^{1/p}\right)\right\}$$
for all
\[\delta >c{{\left( f\left( \frac{{{c}^{{}^{1}/{}_{p}}}p}{{{\delta }^{{}^{1}/{}_{p}}}} \right) \right)}^{p}},\]
 where  $f$ is a function such that  $\varphi(u)=\int_0^u f(v)dv$  $\forall
u>0$ and
where $\varphi^*$ is the Young-Fenchel transform of $\varphi$.
\end{thm}

The following theorem is a direct corollary of the previous one.

\begin{thm} \label{thm2}
Let $\{{\mathbb T},\Lambda, M\}$ be a measurable space, let $X=\{X(t),t\in {\mathbb T}\} \subset S u b_\varphi (\Omega)$, and let $X_N$ be a model of the process $X$.
Let $f$  be a function such that $\varphi(u)=\int_0^u f(v)dv$ $\forall u>0$. Let
$$c_N=\int_{\mathbb T} (\tau_\varphi(X(t)-X_N(t)))^pd\mu(t)<\infty.$$
The model $X_N(t)$ approximates the process $X(t)$ with reliability $1-\alpha$ and accuracy$\delta$ in the $L_p(T)$ space, if
\begin{equation}\label{eq1}
c_N\leq \frac{\delta}{\left(\varphi^{*(-1)}\left(\ln \frac{2}{\alpha}\right)\right)^p}
\end{equation}
and
\begin{equation}\label{eq2}
\delta> c_N\left(f\left( \frac {c_N^{1/p}p}{\delta^{1/p}}\right)\right)^p,
\end{equation}
where $\varphi^*$ is a Young-Fenchel transform of the function $\varphi$.
\end{thm}

\begin{proof}
Since $|X(t)-X_N(t)|\in Sub_\varphi(\Omega)$, Theorem \ref{thm1} implies
$$P\left\{\int_{\mathbb T} |X(t)|^p d\mu(t) >\delta\right\} \leq 2\exp\left\{-\varphi^* \left(\left(\frac{\delta}{c}\right)^{1/p}\right)\right\},$$
therefore,
$$P\left\{\int_{\mathbb T} |X(t)-X_N(t)|^p d\mu(t) >\delta\right\} \leq 2\exp\left\{-\varphi^* \left(\left(\frac{\delta}{c_N}\right)^{1/p}\right)\right\},$$
so inequality \eqref{eq2} has to be true and
$$2\exp\left\{-\varphi^* \left(\left(\frac{\delta}{c_N}\right)^{1/p}\right)\right\}\leq \alpha,$$
whence
$$\varphi^* \left(\left(\frac{\delta}{c_N}\right)^{1/p}\right)\geq \ln \frac{2}{\alpha},$$
and, finally,
$$c_N^{1/p}\leq \frac{\delta^{1/p}}{\varphi^{*(-1)}(\ln \frac{2}{\alpha})}.$$
\end{proof}

\section{Estimation of reliability and accuracy of modelling of stochastic process in $L_p(T)$  spaces}

Based on the Theorem \ref{thm2}, for functions $\varphi(t)$ of specified forms the following theorems can be formulated.

\begin{thm} \label{thm3}
  Let a stochastic process $X=\{X(t),t\in [0,T]\}$ belong to the space $S u b_\varphi (\Omega)$, let
$$\varphi(t)=\frac{t^\gamma}{\gamma}$$
for $1<\gamma\leq 2$.
 Let
$$c_N=\int_0^T (\tau_\varphi(X(t)-X_N(t)))^pd\mu(t)<\infty.$$
 A model $X_N(t)$ approximates the process $X(t)$ with reliability $1-\alpha$ and accuracy $\delta$ in the $L_p(T)$ space, if
$$\left\{\begin{array}{c}c_N\leq \delta/(\beta \ln \frac{2}{\alpha})^{p/\beta}\\ c_N<\delta /p^{p\left(1-1/\gamma\right)}\end{array},\right.$$
where $\beta$ is a number such that $$\frac{1}{\beta}+\frac{1}{\gamma}=1.$$
\end{thm}

\begin{proof}
The first inequality follows from Theorem \ref{thm2} right away.
Indeed, as $\varphi(t)=t^\gamma/\gamma$, then $\varphi^*(t)=t^\beta/\beta$, which gives us the result required.

Since $\alpha\in(0,1)$, then $\ln\frac{2}{\alpha}>0.$

Let us consider the second inequality. Again, since $\varphi(t)=\frac{t^\gamma}{\gamma}$, then $f(t)=t^{\gamma-1}$, $t>0$, and
$$\delta>c_N\left(\left(\frac{c_N^{1/p}p}{\delta^{1/p}}\right)^{\gamma-1}\right)^p = \frac{c_N^\gamma p^{p(\gamma-1)}}{\delta^{\gamma-1}},$$
whence
$$c_N^\gamma<\frac{\delta^\gamma}{p^{p(\gamma-1)}}.$$
\end{proof}

\begin{thm} \label{thm4}
 Let a stochastic process $X=\{X(t),t\in [0,T]\}$ belong to $S u b_\varphi (\Omega)$, let
$$\varphi(t)=\left\{\begin{array}{c}\frac{t^2}{\gamma}, t<1\\ \frac{t^\gamma}{\gamma},t\geq1\end{array}\right.,$$
where $\gamma>2$.
Let
$$c_N=\int_0^T (\tau_\varphi(X(t)-X_N(t)))^pd\mu(t)<\infty.$$

A model $X_N(t)$ approximates the process $X(t)$ with reliability $1-\alpha$ and accuracy$\delta$ in the $L_p(T)$ space, if
$$\left\{\begin{array}
{c}c_N\leq \delta/(\beta \ln \frac{2}{\alpha})^{p/\beta} \\ c_N<\delta/p^{p(1-1/\gamma)}
\end{array}\right.,$$
and $\beta$ is a number such that $$\frac{1}{\beta}+\frac{1}{\gamma}=1.$$
\end{thm}

\begin{proof}
 For the function$ f(t)$ we have
$$f^{(-1)}(t)=\left\{\begin{array}{c}\frac{\gamma}{2}t, t<\frac{2}{\gamma}\\ t^\frac{1}{\gamma-1},t\geq1\end{array}\right..$$
Let us consider $\varphi^*(t)$. If $t>1$, we have
$$\varphi^*(t) = \int_0^{2/\gamma}\frac{\gamma}{2}udu+\int_{2/\gamma}^1du +\int_1^t u^\frac{1}{\gamma-1}du = \left.\frac{\gamma}{2}\frac{u^2}{2}\right|
_0^{\gamma/2}+(1-\frac{2}{\gamma})+\left.u^{\frac{1}{\gamma-1}+1}\right|_1^t =$$
$$= \frac{\gamma}{2}\frac{1}{2}\left(\frac{2}{\gamma}\right)^2+1-\frac{2}{\gamma}+\frac{t^\beta}{\beta}-\frac{1}{\beta} = \frac{t^\beta}{\beta}.$$

This result and Theorem \ref{thm2} induce the first inequality of this Theorem  \ref{thm4}.

Let us consider the second inequality now. Let us first analyze the case when $$\frac{c_N^{1/p}p}{\delta^{1/p}}>1.$$
In such a case, $f(x)=x^{\gamma-1}$, and
$$\delta>c_N\left(\left(\frac{c_N^{1/p}p}{\delta^{1/p}}\right)^{\gamma-1}\right)^p,$$
that is,
$$c_N<\frac{\delta}{p^{p(1-1/\gamma)}},$$
whence
$$\frac{\delta}{p^p}<c_N<\frac{\delta}{p^{p(1-1/\gamma)}}.$$

For the case $$\frac{c_N^{1/p}p}{\delta^{1/p}}>1$$ we have $$f(x)=x\frac{2}{\gamma},$$ and
$$\left\{\begin{array}{c} c_N<\frac{\delta}{p^p}\\ c_N<\frac{\delta}{p^p/2}\left(\frac{\gamma}{2}\right)^{p/2}. \end{array}\right.$$
Since $\gamma>2$ and $p>1$, we obtain
$$c_N<\frac{\delta}{p^p}.$$
Finally, we have
$$c_N<\frac{\delta}{p^{p(1-1/\gamma)}}.$$
\end{proof}

\section{Construction models of stochastic processes from $Sub_\varphi(\Omega)$ that can be represented as a series with independent elements.}

Assume we can represent a stochastic process $X = \{X(t),t\in [0,T]\}$ in the form of series
\begin{equation}\label{eq3}
X(t) = \sum_{k=1}^\infty \xi_k a_k(t),
\end{equation}
where $\xi_k\in Sub_\varphi(\Omega)$, $\xi_k$ are independent, and the next property is true for this series:
$$\sum_{k=1}^\infty \tau_\varphi^2(\xi_k)a^2_k(t)<\infty.$$

Usually, a sum of first $N$ elements of this representation is used as a model of such a process.
However, functions $a_k(t)$ can often be unfindable explicitly. In this case $\hat{a}_k(t)$, approximations of function $a_k(t)$,
may be used as elements of the model of a stochastic process after taking into account the impact of
such approximation an accuracy and reliability of the process approximation with the model.

Definition 6. We will call a model of a stochastic process $X(t)$ the following expression:
$$X_N(t) = \sum_{k=1}^N \xi_k \hat{a}_k(t),$$
where $\hat{a}_k(t)$ are approximations of function $a_k(t)$, $\xi_k\in Sub_\varphi(\Omega)$, $\xi_k$ are independent.

Let us introduce the following notations:
$$\delta_k(t) = |a_k(t)-\hat{a}_k(t)|;$$
$$\Delta_N(t) = |X(t)-X_N(t)| = \left| \sum_{k=1}^N \xi_k \delta_k(t) + \sum_{k=N+1}^\infty \xi_k a_k(t) \right|.$$

We will say that a model $X_N$ approximates a stochastic process $X$ with given accuracy and reliability in the space $L_p[0,T]$, if
$$P\left\{\left(\int_0^T(\Delta_N(t))^pdt\right)^{\frac{1}{p}}>\delta\right\} \leq \alpha.$$
Let us formulate a theorem for simulation of such processes in $L_p(T)$.

\begin{thm} \label{thm5}
 Let s stochastic process $X=\{X(t),t\in [0,T]\}$ belong to $S u b_\varphi (\Omega)$, let $X_N$ be a model of the process $X$. Assume
$$c_N=\int_0^T\left(\tau_\varphi\left(\sum_{k=1}^N \xi_k\delta_k(t) + \sum_{k=N+1}^\infty\xi_k a_k(t)  \right)\right)^{p} dt<\infty.$$
The model $X_N(t)$approximates the stochastic process $X(t)$ with reliability $1-\alpha$ and accuracy $\delta$ in the space $L_p(T)$, when
$$\left\{ \begin{array}{c}c_N\leq \delta/(\varphi^{*(-1)}(ln \frac{2}{\alpha}))^p\\ \delta> c_N(f( \frac {c_N^{1/p}p}{\delta^{1/p}}))^p\end{array}\right.,$$
the function $f$ is provided in Theorem \ref{thm2}, $\varphi^*$ is a Young-Fenchel transform of $\varphi$.
\end{thm}

\begin{proof}
The proof of this Theorem \ref{thm5} is similar to the proof of Theorem \ref{thm2}.
\end{proof}

For theorems that will follow we require the following statement.

\begin{thm} \label{thm6} \cite{Buldygin}
Let $\xi_1, \xi_1,\ldots,\xi_n\in Sub_\varphi(\Omega)$ be independent random variables.
If a function $\varphi(|x|^{1/s})$, $x\in R$ is convex for $s\in(0,2]$, then
$$\tau_\varphi^s\left(\sum_{k=1}^n\xi_k\right)\leq \sum_{k=1}^n \tau_\varphi^s(\xi_k).$$
\end{thm}

\begin{thm} \label{thm7}
 Let a stochastic process $X=\{X(t),t\in [0,T]\}$ belong to $S u b_\varphi (\Omega)$,
$$\varphi(t)=\left\{\begin{array}{c}\frac{t^2}{\gamma}, t<1\\ \frac{t^\gamma}{\gamma},t\geq1\end{array}\right.,$$
where  $\gamma>2$.
Assume
$$c_N=\int_0^T\left(\sum_{k=1}^N \tau^2_\varphi(\xi_k)\delta^2_k(t) + \sum_{k=N+1}^\infty\tau^2_\varphi(\xi_k)a^2_k(t)
\right)^{p/2} dt<\infty.$$
The model $X_N(t)$ approximates the stochastic process $X(t)$ with reliability $1-\alpha$ and accuracy $\delta$ in the space $L_p(T)$, if
$$\left\{\begin{array}{c}c_N\leq \delta/(\beta \ln \frac{2}{\alpha})^{p/\beta}\\ c_N<\delta/p^{p(1-1/\gamma)}\end{array}\right.,$$
where $\beta$ is such that $1/\beta+1/\gamma=1$.
\end{thm}

\begin{proof}  From Theorem \ref{thm6} for $s=2$ the next inequalities follow:
$$c_N=\int_0^T
(\tau_\varphi(\Delta_N(t)))^pd\mu(t) = $$
$$=\int_0^T\left( \tau_\varphi\left(\sum_{k=1}^N\xi_k\delta_k(t)+
\sum_{k=N+1}^\infty \xi_ka_k(t)\right) \right)^pdt \leq$$
$$
\leq\int_0^T\left(\sum_{k=1}^N \tau^2_\varphi(\xi_k)\delta^2_k(t) + \sum_{k=N+1}^
\infty\tau^2_\varphi(\xi_k)a^2_k(t)  \right)^{p/2} dt.$$
The last inequality follows from Theorem \ref{thm6} and properties of the function $\tau_\varphi$.
\end{proof}

\begin{thm} \label{thm8}
Let a stochastic process $X=\{X(t),t\in [0,T]\}$ belong to $S u b_\varphi (\Omega)$,
$\varphi(t)= \frac{t^\gamma}{\gamma},$
where  $1<\gamma<2$. Assume
$$c_N=\int_0^T\left(\sum_{k=1}^N \tau^{\gamma}_\varphi(\xi_k)\delta^{\gamma}_k(t) + \sum_{k=N+1}^\infty\tau^{\gamma}_
\varphi(\xi_k)a^{\gamma}_k(t)  \right)^{p/\gamma} dt<\infty.$$
The model $X_N(t)$ approximates the stochastic process $X(t)$ with reliability $1-\alpha$ and accuracy $\delta$ in the space $L_p(T)$, if
$$\left\{\begin{array}{c}c_N\leq \delta/(\beta ln \frac{2}{\alpha})^{p/\beta}\\ c_N<\delta/p^{p(1-1/\gamma)}\end{array}\right.,$$
where $\beta$ is such that $1/\beta+1/\gamma=1$.
\end{thm}

\begin{proof}
Theorem \ref{thm6} for $s=\gamma$ implies the following inequalities:
$$c_N=\int_0^T
(\tau_\varphi(\Delta_N(t)))^pd\mu(t) = $$
$$=\int_0^T\left( \tau_\varphi\left(\sum_{k=1}^N\xi_k\delta_k(t)+
\sum_{k=N+1}^\infty \xi_ka_k(t)\right) \right)^pdt \leq$$
$$\leq\int_0^T\left(\sum_{k=1}^N \tau^\gamma_\varphi(\xi_k)\delta^\gamma_k(t) + \sum_
{k=N+1}^\infty\tau^\gamma_\varphi(\xi_k)a^\gamma_k(t)  \right)^{p/\gamma} dt.$$
The last inequality follows from Theorem \ref{thm6} and properties of the function $\tau_\varphi$.
\end{proof}

\section{ Constructing models of stochastic processes in $L_p(0,T)$ using the Karhunen-Lo{\`e}ve decomposition}

As an example of application of the Theorems formulated above, we  consider a well-known Karhunen-Lo{\`e}ve decomposition of stochastic processes.
Let $X=\{X(t), t\in [0,T]\}$  be continuous in a mean square stochastic process, $EX(t)=0$, $t\in T$,
let $B(t,s)=E X(t)X(s)$ be the correlation function of this stochastic process.
Sine the process $X(t)$ is continuous in the mean square, the function $B(t,s)$ is continuous on $T\times T$.

Consider the homogenous Fredholm integral equation of the second kind
\begin{equation}\label{eq4}
a(t) = \lambda \int_T B(t,s) a(s) ds.
\end{equation}
It is well-known (see \cite{Tricomi}) that this integral equation has at most countable family of eigenvalues.
These numbers are non-negative. Let $\lambda _{k}^{2}$ be the eigenvalues, and let ${{a}_{k}}(t)$ be the corresponding eigenfunctions.
Let us enumerate the eigenvalues in the ascending order: $$0\le {{\lambda }_{1}}\le {{\lambda }_{2}}\le \ldots \le {{\lambda }_{k}}\le {{\lambda }_{k+1}}\le \ldots .$$
It is well-known that ${{a}_{k}}(t)$ are orthogonal functions.

The main issue while constructing a Karhunen-Lo{\`e}ve model is the fact that eigenvalues and eigenfunctions of such integral equations
can be found explicitly only for a few correlation functions.
Let $\hat{a}_k$ be functions that are the approximations of the eigenfunctions of the integral equation,
$\hat{\lambda}_k$ be approximations of the corresponding eigenvalues of this equation.
We can construct such approximations using, for example, a method described in \cite{mokl3}.

Definition 7. Let $X=\{X(t), t\in [0,T]\}\subset Sub_\varphi(\Omega)$ be a stochastic process with the correlation function $B(t,s)=EX(t)\overline{X(s)}$
that can be represented using the Karhunen-Lo{\`e}ve decomposition.
We will call the process $X_N=\{X_N(t),t\in [0,T]\} \subset Sub_\varphi
(\Omega)$ the Karhunen-Lo\'eve model of the mentioned process, if
$$X_N(t)=\sum_{k=1}^N \frac{\hat{a}_k(t)}{\sqrt{\hat{\lambda}}_k}\xi_k,$$
where $\hat{a}_k(t)$ are approximations of eigenfunctions of the homogenous integral Fredholm equation of second kind \eqref{eq4},
and $\hat{\lambda}_k$ are approximations of eigenvalues of the same equation.

Let us formulate theorems that will allow to construct models of stochastic processes using this decomposition in $L_p(T)$.

\begin{thm} \label{thm9}
 Let a stochastic process $X=\{X(t),t\in [0,T]\}$ belong to $S u b_\varphi (\Omega)$, and let
$$\varphi(t)=\left\{\begin{array}{c}\frac{t^2}{\gamma}, t<1\\ \frac{t^\gamma}{\gamma},t\geq1\end{array}\right.$$
for $\gamma>2$, and can be represented using the Karhunen-Lo{\`e}ve decomposition. Assume
$$c_N= \int_0^T\left(\sum_{k=1}^N \tau^2_\varphi(\xi_k)\left
(\frac{\delta^2_k(t) }{\hat{\lambda}_k-\eta_k} +\hat{a}^2_k(t)\frac{(\sqrt{\hat{\lambda}_k}-\sqrt{\hat{\lambda}_k-\eta_k})^2}{\hat{\lambda}_k(\hat
{\lambda}_k-\eta_k)}\right)+\right.$$
$$\left.+\sum_{k=N+1}^\infty\frac{\tau^2_\varphi(\xi_k)a^2_k(t)}{\lambda_k} \right)^{p/2}dt<\infty,$$
where $\delta_k(t)=|\varphi_k(t)-\hat{\varphi}_k(t)|$ is the error of approximation of $k$-th eigenfunction of equation \eqref{eq4},
$\hat{\lambda}_k$ is the approximation of the $k$-the eigenvalue,
$\eta_k=|\lambda_k-\hat{\lambda}_k|$ is the error of approximation of the $k$-the eigenvalue.
The model $X_N(t)$ approximates the stochastic process $X(t)$ with reliability $1-\alpha$ and accuracy $\delta$ in the space $L_p(T)$,
when the following conditions are true:
$$\left\{\begin{array}{c}c_N\leq \delta/(\beta ln \frac{2}{\alpha})^{p/\beta} \\ c_N<\delta/p^{p(1-1/\gamma)}
\end{array}\right..$$
\end{thm}

\begin{proof}
The statement of the theorem follows from Theorem \ref{thm7}.
Indeed,
$$c_N=\int_0^T\left(\tau_\varphi\left(\sum_{k=1}^N \xi_k\left(\frac{a_k(t)}{\sqrt{\lambda_k}}-\frac{\hat{a}_k(t)}{\sqrt{\hat{\lambda}_k}}\right) +
\sum_{k=N+1}^\infty\xi_k\frac{a_k(t)}{\sqrt{\lambda_k}}  \right)\right)^{p} dt =$$
$$= \int_0^T\left(\tau_\varphi\left(\sum_{k=1}^N \xi_k\left(\frac{a_k(t)}{\sqrt{\lambda_k}}-\frac{\hat{a}_k(t)}{\sqrt{\lambda_k}} + \frac{\hat{a}_k
(t)}{\sqrt{\lambda_k}} -\frac{\hat{a}_k(t)}{\sqrt{\hat{\lambda}_k}}\right) + \right.\right.$$
$$\left.\left.+ \sum_{k=N+1}^\infty\xi_k\frac{a_k(t)}{\sqrt{\lambda_k}}  \right)\right)^{p} dt= \int_0^T\left(\tau_\varphi\left(\sum_{k=1}^N \xi_k\frac{1}{\sqrt{\lambda_k}}\delta_k(t) +\right.\right.$$
$$\left.\left.+\sum_{k=1}^N\xi_k\hat{a}_k(t)\left(\frac{1}{\sqrt{\lambda_k}}-\frac{1}{\sqrt{\hat{\lambda}_k}}\right)  + \sum_{k=N+1}^\infty\tau_\varphi
(\xi_k)\frac{a_k(t)}{\sqrt{\lambda_k}}  \right)\right)^{p} dt\leq
$$
$$\leq \int_0^T\left(\sum_{k=1}^N \tau^2_\varphi(\xi_k)\left(\frac{\delta^2_k(t) }{\hat{\lambda}_k-\eta_k} +\hat{a}^2_k(t)\frac{(\sqrt{\hat{\lambda}
_k}-\sqrt{\hat{\lambda}_k-\eta_k})^2}{\hat{\lambda}_k(\hat{\lambda}_k-\eta_k)}\right)+\right.$$
$$\left.+\sum_{k=N+1}^\infty\frac{\tau^2_\varphi(\xi_k)a^2_k(t)}{\lambda_k} \right)^{p/2}dt.$$
In the case when $\tau_\varphi(\xi_k)\leq\tau$ $\forall k$ we will have the next statement.
\end{proof}

\begin{thm} \label{thm10}
 Let a stochastic process $X=\{X(t),t\in [0,T]\}$ belong to $S u b_\varphi (\Omega)$,
$$\varphi(t)=\left\{\begin{array}{c}\frac{t^2}{\gamma}, t<1\\ \frac{t^\gamma}{\gamma},t\geq1\end{array}\right.$$
for  $\gamma>2$, and let $\forall k:$ $\tau_\varphi(\xi_k)=\tau$.
Assume the process $X$ can be represented using the Karhunen-Lo\`eve decomposition, and let
$$c_N=\tau^{p/2}\int_0^T\left(\left(B(t,t)-\sum_{k=1}^N \frac{(\hat{a}_k(t)-\delta_k(t))^2}{\hat{\lambda}_k+\eta_k}\right)\right.+ $$
$$+\left.\sum_{k=1}^N \left
(\frac{\delta^2_k(t) }{\hat{\lambda}_k-\eta_k} +
\hat{a}^2_k(t)\frac{(\sqrt{\hat{\lambda}_k}-\sqrt{\hat{\lambda}_k-\eta_k})^2}{\hat{\lambda}_k(\hat{\lambda}_k-\eta_k)}\right)\right)^
{p/2} dt<\infty,$$
where $\delta_k(t)=|\varphi_k(t)-\hat{\varphi}_k(t)|$ is the error of approximation of $k$-th eigenfunction of equation \eqref{eq4},
$\hat{\lambda}_k$ is the approximation of the $k$-the eigenvalue, $\eta_k=|\lambda_k-\hat{\lambda}_k|$ is the error of approximation of the $k$-the eigenvalue.
The model $X_N(t)$ approximates the stochastic process $X(t)$ with reliability $1-\alpha$ and accuracy $\delta$ in the space $L_p(T)$, when the following conditions are true:
$$\left\{\begin{array}{c}c_N\leq \delta/(\beta ln \frac{2}{\alpha})^{p/\beta} \\ c_N<\delta/p^{p(1-1/\gamma)}\end{array}\right..$$
\end{thm}

\begin{proof}
This statement follows from the Mercer's Theorem \cite{Tricomi} and the previous Theorem.
Indeed, according to the Mercer's Theorem,
$$B(t,t)=\sum_{k=1}^\infty \frac{a^2_k(t)}{\lambda_k},$$
therefore
$$c_N= \int_0^T\left(\sum_{k=1}^N \tau^2\left(\frac{\delta^2_k(t) }{\hat{\lambda}_k-\eta_k} +\hat{a}^2_k(t)\frac{(\sqrt{\hat{\lambda}_k}-\sqrt{\hat
{\lambda}_k-\eta_k})^2}{\hat{\lambda}_k(\hat{\lambda}_k-\eta_k)}\right)+\sum_{k=N+1}^\infty\frac{\tau^2a^2_k(t)}{\lambda_k} \right)^{p/2}dt=
$$
$$=\tau^{p/2}\int_0^T\left(\sum_{k=1}^N \tau^2\left(\frac{\delta^2_k(t) }
{\hat{\lambda}_k-\eta_k} +
\hat{a}^2_k(t)\frac{\left(\sqrt{\hat{\lambda}_k}-\sqrt{\hat{\lambda}_k-\eta_k}\right)^2}{\hat{\lambda}_k(\hat{\lambda}_k-\eta_k)}\right)
+ B(t,t) - \sum_{k=1}^N\frac{a^2_k(t)}{\lambda_k} \right)^{p/2}dt\leq$$
$$\leq\tau^{p/2}\int_0^T\left(\left(B(t,t)-\sum_{k=1}^N \frac{(\hat{a}_k(t)-\delta_k(t))^2}{\hat{\lambda}_k+\eta_k}\right)\right.+
$$
$$\left.+\sum_{k=1}^N \left(\frac
{\delta^2_k(t) }{\hat{\lambda}_k-\eta_k} +
\hat{a}^2_k(t)\frac{(\sqrt{\hat{\lambda}_k}-\sqrt{\hat{\lambda}_k-\eta_k})^2}{\hat{\lambda}_k(\hat{\lambda}_k-\eta_k)}\right)\right)^
{p/2} dt.$$
\end{proof}

\begin{thm} \label{thm11}
Let a stochastic process $X=\{X(t),t\in [0,T]\}$ belong to $S u b_\varphi (\Omega)$, let
$$\varphi(t)= \frac{t^\gamma}{\gamma}$$
for $1<\gamma<2$.
Assume
$$c_N= \int_0^T\left(\sum_{k=1}^N \tau^{\gamma}_\varphi(\xi_k)\left(\frac{\delta^{\gamma}_k(t) }{(\lambda_k-
\eta_k)^\gamma} +\hat{a}^{\gamma}_k(t)\frac{(\sqrt{\hat{\lambda}_k}-\sqrt{\hat{\lambda}_k-\eta_k})^{\gamma}}{(\hat{\lambda}_k(\hat{\lambda}_k-\eta_k))^
\gamma}\right)+\right.$$
$$+\left.\sum_{k=N+1}^\infty\frac{\tau^{\gamma}_\varphi(\xi_k)a^{2\gamma}_k(t)}{\lambda_k^\gamma} \right)^{p/\gamma}dt,$$
where $\delta_k(t)=|\varphi_k(t)-\hat{\varphi}_k(t)|$ is the error of approximation of $k$-th eigenfunction of equation (5), $\hat{\lambda}_k$ is the approximation of the $k$-the eigenvalue, $\eta_k=|\lambda_k-\hat{\lambda}_k|$ is the error of approximation of the $k$-the eigenvalue.
The model $X_N(t)$ approximates the stochastic process $X(t)$ with reliability $1-\alpha$ and accuracy $\delta$ in the space $L_p(T)$, when the following conditions are true:
$$\left\{ \begin{matrix}
   {{c}_{N}}\le \delta /{{(\beta ln\frac{2}{\alpha })}^{p/\beta }}  \\
   {{c}_{N}}<\delta /{{p}^{p(1-1/\gamma )}}  \\
\end{matrix} \right..$$
\end{thm}

\begin{proof}
he statement of this theorem follows from the Theorem \ref{thm8} and the proof of the Theorem \ref{thm10}.
\end{proof}

\section{{Conclusions}}

This article discusses modeling methods with specified reliability and accuracy of stochastic processes from $Sub_\varphi(\Omega)$
spaces in $L_p(T)$ spaces that allow decomposition in series with independent elements.
Theorems are proved that allow constructing models of such stochastic processes with given reliability and accuracy.
Besides, theorems have been proven for stochastic processes with expansion terms that cannot be explicitly determined but can be approximated by certain functions.
Similar theorems for some $Sub_\varphi(\Omega)$ spaces with given functions $\varphi$. As an example of the application of the results of the article,
theorems have been proven that allow constructing the Karhunen-Lo{\`e}ve model in cases where the integral equation corresponding to the model cannot be explicitly solved.

\noindent\hrulefill


\begin{thebibliography}{99}

 {

 \bibitem{Buldygin}
V.V. Buldygin, Yu.V. Kozachenko. Metric Characterization of Random Variables and Random Processes.  -- Providence: AMS, 260 p.--2000.

\bibitem{kozach}
Yu.V. Kozachenko, A.O. Pashko. Modelyuvannya vypadkovykh protsesiv, {Ky\"{\i}v: Vydavnychyj Tsentr ``Ky\"{\i}vs'kyj Universytet''}, 223 p.-- 1999.

\bibitem{koz-kam}
Yu. V. Kozachenko, O. E. Kamenshchikova.
Approximation of $\text{SSub}_{\varphi}(\Omega)$  stochastic processes in the space $L_p(T)$.
Theory Probab. Math. Stat. 79, 83--88, 2009.

\bibitem{Tricomi}
F. G. Tricomi. Integral equations. --  New York: Interscience Publishers, 238p.--1957.

\bibitem{mokl3}
O.M. Mokliachuk.  Simulation of random processes with known correlation function with the help of Karhunen-Lo{\`e}ve decomposition,
Theory of Stochastic Processes, Vol. 13, No.4, 163-169, 2007
}

\end{thebibliography}
\end{document}